\documentclass[12pt]{amsart}
\usepackage{amsmath, amsthm}
\title[Invariant submanifolds of Lorentzian trans-Sasakian manifolds... \ldots]
{Tachibana operator Applied to Invariant submanifolds of Lorentzian trans-Sasakian manifolds}

\author{Mehmet At\d {c}eken and Tugba Mert }
\address{Aksaray University, Faculty of Arts and Sciences, Department of Mathematics,  68100, Aksaray, TURKEY}
\email{mehmet.atceken382@gmail.com}
\address{Sivas Cumhuriyet University, Faculty of Sciences, Department of Mathematics,  58100, Sivas, TURKEY}
\email{tmert@cumhuriyet.edu.tr}


\newtheorem{theorem}{Theorem}[section]

\newtheorem{example}[theorem]{Example}

\newtheorem{proposition}[theorem]{Proposition}

\begin{document}
    \maketitle
\begin{abstract}
In the present paper, Tachibana operatory is applied to an invariant submanifold of a lorentzian trans-Sasakian manifold by means of through various tensors and the results obtained are discussed in terms of geometry. Finally, we give a non-trivial example in order to our results illustrate.
\end{abstract}
\section{Introduction}

An differentiable manifold $\widetilde{M}^{2n+1}$ which carries a
field $\phi$ of endomorphism of the tangent space, a timelike vector
$\xi$, called characteristic vector field, a 1-form $\eta$ and the
lorentzian metric $g$ satisfying
\begin{eqnarray}
\phi^2=I+\eta\otimes\xi, \ \ \eta(\xi)=-1,\label{1}
\end{eqnarray}
\begin{eqnarray}
g(\phi{X},\phi{Y})=g(X,Y)+\eta(X)\eta(Y), \ \
\eta(X)=g(X,\xi)\label{2}
\end{eqnarray}
and
\begin{eqnarray}
(\widetilde{\nabla}_X\phi)Y=\alpha\{g(X,Y)\xi-\eta(Y)X\}+\beta\{g(\phi{X},Y)-\eta(Y)\phi{X}\}\label{3},
\end{eqnarray}
for all $X,Y\in\Gamma(T\widetilde{M})$\cite{1}, where $\widetilde{\nabla}$
denote the Levi-Civita connection and $\alpha$, $\beta$ are smooth
functions on $\widetilde{M}$. From (\ref{2}), we have
\begin{eqnarray}
\widetilde{\nabla}_X\xi=-\alpha\phi{X}-\beta\phi^2{X},\label{4}
\end{eqnarray}
\begin{eqnarray}
(\widetilde{\nabla}_X\eta)Y=\alpha{g}(\phi{X},Y)+\beta{g}(\phi{X},\phi{Y}).\label{5}
\end{eqnarray}

As special cases, if $\alpha=0$ and $\beta\in\textbf{R}$ the set of
real numbers, then the manifold reduces to a Lorentzian
$\beta$-Kenmotsu manifold, if $\beta=0$ and $\alpha\in\textbf{R}$,
real number, then the manifold reduces to a Lorentzian
$\alpha$-Sasakian manifold. On the other hand, $\alpha=0$ and
$\beta=1$, the manifold reduces Lorentzian Kenmotsu manifold
introduced by Mihai Oiaga and Rosca\cite{7}. Furthermore, if
$\beta=0$ and $\alpha=1$, then the manifold reduces to a Lorentzian Sasakian manifold invariant submanifolds have been studied in different structures and working on\cite{7,4,5,6}.\\

In the present paper, we
moved to lorentzian trans-Sasakain manifold and research cases of $Q(S,\sigma)=0$, $Q(g,\sigma)=0$ $Q(g,R\cdot{\sigma})=0$, $Q(g,\widetilde{\nabla}\cdot{\sigma})=0$, $Q(S,\sigma)=0$,
$Q(S,\nabla\sigma)=0$, $Q(S,R\cdot\sigma)=0$,
$Q(g,C\cdot\sigma)=0$ and $Q(S,\widetilde{C}\cdot\sigma)=0$ for an invariant submanifold of lorentzian trans-Sasakian manifold. In these cases, we examine the properties reduced by both the ambient manifold and the necessary submanifold.\\

In this connection, we need the following proposition for later
used.
\begin{proposition}
Let $\widetilde{M}^{2n+1}(\phi,\xi,\eta,g)$ an lorentzian
trans-Sasakain manifold, we denote the Riemannian curvature and
Ricci tensors by $\widetilde{R}$ and $\widetilde{S}$, respectively,
then we have
\begin{eqnarray}
\widetilde{R}(X,Y)\xi&=&(\alpha^2+\beta^2)[\eta(Y)X-\eta(X)Y]+2\alpha\beta[\eta(Y)\phi{X}-\eta(X)\phi{Y}]\nonumber\\
&+&Y(\alpha)\phi{X}-X(\alpha)\phi{Y}+Y(\beta)\phi^2X-X(\beta)\phi^2Y,\label{6}\\
&&\eta(\widetilde{R}(X,Y)Z)=(\alpha^2+\beta^2)g(\eta(Y)X-\eta(X)Y,Z)\label{7}\\
&&\widetilde{R}(\xi,X)\xi=(\alpha^2+\beta^2-\xi(\beta))\phi^2X+(2\alpha\beta-\xi(\alpha))\phi{X},\label{8}\\
&&\widetilde{S}(X,\xi)=[2n(\alpha^2+\beta^2)-\xi(\beta)]\eta(X)+(2n-1)X(\beta)\nonumber\\
&-&(\phi{X})(\alpha)\label{9}
\end{eqnarray}
\end{proposition}
for all vector fields $X,Y$ on $\widetilde{M}^{2n+1}$\cite{1}.\\

Now, let $M$ be an immersed submanifold of a lorentzian
trans-Sasakian manifold $\widetilde{M}^{2n+1}$. By $\Gamma(TM)$ and
$\Gamma(T^\bot{M})$, we denote the tangent and normal subspaces of
$M$ in $\widetilde{M}$. Then the Gauss and Weingarten formulae are,
respectively, given by
\begin{eqnarray}
\widetilde{\nabla}_XY=\nabla_XY+\sigma(X,Y),\label{10}
\end{eqnarray}
and
\begin{eqnarray}
\widetilde{\nabla}_XV=-A_VX+\nabla^\bot_XV,\label{11}
\end{eqnarray}
for all $X,Y\in\Gamma(TM)$ and $V\in\Gamma(T^\bot{M})$, where
$\nabla$ and $\nabla^\bot$  are the induced connections on $M$ and
$\Gamma(T^\bot{M})$ and $\sigma$ and $A$ are called the second
fundamental form and shape operator of $M$, respectively. Also,
$\Gamma(TM)$ denotes the set differentiable vector fields on $M$.
They are related by
\begin{eqnarray}
g(A_VX,Y)=g(\sigma(X,Y),V). \label{12}
\end{eqnarray}
The covariant derivative of $\sigma$ is defined by
\begin{eqnarray}
(\widetilde{\nabla}_X\sigma)(Y,Z)=\nabla^\bot_X\sigma(Y,Z)-\sigma(\nabla_XY,Z)-\sigma(Y,\nabla_XZ),\label{13}
\end{eqnarray}
for all $X,Y,Z\in\Gamma(TM)$\cite{4}. If
$\widetilde{\nabla}\sigma=0$, then submanifold is said to be its
second fundamental form is parallel. On the other hand, the
submanifold $M$ is said to be Chaki pseudo-parallel if there exists
a 1-form $\gamma$ such that
    \begin{eqnarray}
    (\widetilde{\nabla}_X\sigma)(Y,Z)&=&2\gamma(X)\sigma(Y,Z)+\gamma(Y)\sigma(X,Z)\nonumber\\
    &+&\gamma(Z)\sigma(X,Y),\label{14}
\end{eqnarray}
for all $X,Y\in(TM)$. \\

By $R$, we denote the Riemannian curvature tensor of the submanifold
$M$, we have the following Gauss equation
\begin{eqnarray}
\widetilde{R}(X,Y)Z&=&R(X,Y)Z+A_{\sigma(X,Z)}Y-A_{\sigma(Y,Z)}X+(\widetilde{\nabla}_X\sigma)(Y,Z)\nonumber\\
&-&(\widetilde{\nabla}_Y\sigma)(X,Z)\label{15},
\end{eqnarray}
for all $X,Y,Z\in\Gamma(TM)$.\\

On the other hand, the concircular curvature tensor $\mathcal{C}$ on a pseudo-Riemannian manifold $(M^{2n+1},g)$ is
defined as follows
\begin{eqnarray}
\mathcal{C}(X,Y)Z=R(X,Y)Z-\dfrac{\tau}{2n(2n+1)}\{g(Y,Z)X-g(X,Z)Y\},\label{16}
\end{eqnarray}
for all $X,Y,Z\in\Gamma(TM)$\cite{2}, where $\tau$ denotes the scalar curvature of $M^{2n+1}$.\\

For a $(0,k)$-type tensor field $T$, $k\geq{1}$ and a $(0,2)$-type
tensor field $A$ on a Riemannian manifold $(M,g)$, $Q(A,T)$-Tachibana operator
is defined by
\begin{eqnarray}
Q(A,T)(X_1,X_2,...,X_k;X,Y)&=&-T((X\wedge_AY)X_1,X_2,...,X_k)...\nonumber\\
&-&T(X_1,X_2,...X_{k-1},(X\wedge_AY)X_k),\label{17},
\end{eqnarray}
for all $X_1,X_2,...,X_k,X,Y\in\Gamma(TM)$\cite{2}, where the
endormorphism $\wedge_A$ is defined by
\begin{eqnarray}
(X\wedge_AY)Z=A(Y,Z)X-A(X,Z)Y. \label{18}
\end{eqnarray}

\section{Invariant Submanifolds of Lorentzian trans-Sasakain manifolds.}
The geometry of submanifolds of a contact(paracontact, contact and product structures) metric manifold is depend on the behaviour of contact metric structure $\phi$. Namely, a submanifold $M$ of a lorentzian trans-Sasakian manifold is said to be invariant if the structure vector field $\xi$ is tangent to $M$ at every point of $M$ and $\phi{X}$ is tangent to $M$ for any vector feld $X$ tangent to $M$ at every point of $M$. In other words, $\phi(TM)\subset(TM)$ at each point of $M$.\\

    In the submanifolds theory, we note that the geometry of invariant submanifolds inherits almost all properties of the ambient manifolds. Therefore, invariant submanifolds are an active and fruitful research field playing a significant role in the development of modern differential geometry. In this connection, the papers related to invariant submanifolds has been studied and continues to be studied on the different structures.\\

    In the rest of this paper, we will assume that $M$ an invariant submanifold of lorentzian trans-Sasakian manifold $\widetilde{M}$ unless otherwise stated.\\

    So we need the following Theorem for later used.
\begin{theorem}
Let $M$ be an immersed submanifold of a lorentzian trans-Sasakian
manifold $\widetilde{M}(\phi,\xi,\eta,g)$.
By $R$ and $\sigma$, we denote the Riemannian curvature tensor and second fundamental form of submanifold $M$, respectively. Then the following relations hold;\\
\begin{eqnarray}
\widetilde{R}(X,Y)\xi&=&R(X,Y)\xi,\label{19}\\
\sigma(X,\phi{Y})&=&\sigma(\phi{X},Y)=\phi\sigma(X,Y),\\\label{20}
\sigma(X,\xi)&=&A_V\xi=0,\label{21}
\end{eqnarray}
for all $X,Y\in\Gamma(TM)$ and $V\in\Gamma(T^\bot{M})$.
\end{theorem}
\begin{proof}
Since the proofs are a result of simple calculations, we think that to be unnecessary to give.
\end{proof}
\begin{theorem}
	Let $M$ be an invariant submanifold of a lorentzian trans-Sasakian manifold $\widetilde{M}$. If $Q(S,\sigma)=0$, then $M$ is either totally geodesic submanifold or $\xi(\beta)=\alpha^2+\beta^2$.
\end{theorem}
\begin{proof}
$Q(S,\sigma)=0$ means that
\begin{eqnarray}
Q(S,\sigma)(U,V;X,Y)&=&\sigma((X\wedge_SY)U,V)+\sigma(V,(X\wedge_SY)U)\nonumber\\
&=&\sigma(S(Y,U)X-S(X,U)Y,V)\nonumber\\
&+&\sigma(U,S(Y,V)X-S(X,V)Y)=0,\label{22}
\end{eqnarray}
for all $X,Y,U,V\in\Gamma(TM)$. Taking $X=V=\xi$ in (\ref{22}) and by using (\ref{9}), we have
\begin{eqnarray*}
-2n(\alpha^2+\beta^2-\xi(\beta))\sigma(U,Y)=0,
\end{eqnarray*}
which proves our assertion.
\end{proof}
\begin{theorem}
Let $M$ be a invariant submanifold of a lorentzian trans-Sasakian manifold $\widetilde{M}$. If $Q(g,\sigma)=0$ if and only if $M$ is totally geodesic submanifold.
\end{theorem}
\begin{proof}
$Q(g,\sigma)=0$ implies that
\begin{eqnarray*}
Q(g,\sigma)(U,V;X,Y)=\sigma((X\wedge_gY)U,V)+\sigma(U,(X\wedge_gY)V)=0,
\end{eqnarray*}
for all $X,Y,U,V\in\Gamma(TM)$. Substituting $Y=U$=$\xi$ in the last equality, we can conclude
$\sigma(X,V)=0$.
\end{proof}
\begin{theorem}
Let $M$ be an invariant submanifold of a lorentzian trans-Sasakian manifold $\widetilde{M}$. If $Q(S,\widetilde{\nabla}\cdot\sigma)=0$, then at least one of the following holds;\\
1.) $M$ is a totally geodesic,\\
2.) $\xi(\beta)=\alpha^2+\beta^2$,\\
3.) $\alpha^2-\beta^2=0$.
\end{theorem}
\begin{proof}
If $M$ is an invariant submanifold and $Q(S,\nabla\cdot\sigma)=0$, then we have
\begin{eqnarray*}
Q(S,\widetilde{\nabla}\cdot\sigma)(U,V,Z;X,Y)&=&(\widetilde{\nabla}_{(X\wedge_SY)U}\sigma)(V,Z)+(\widetilde{\nabla}_U\sigma)((X\wedge_SY)V,Z)\nonumber\\
&+&(\widetilde{\nabla}_U\sigma)(V,(X\wedge_SY)Z)=0,
\end{eqnarray*}
for all $X,Y,U,V,Z\in\Gamma(TM)$.
For $Y=Z=\xi$, this yields to
\begin{eqnarray}
(\widetilde{\nabla}_{(X\wedge_S\xi)U}\sigma)(V,\xi)&+&(\widetilde{\nabla}_U\sigma)((X\wedge_S\xi)V,\xi)\nonumber\\
&+&(\widetilde{\nabla}_U\sigma)(V,(X\wedge_S\xi)\xi)=0.\label{23}
\end{eqnarray}
Here, non-zero components of the first term have
\begin{eqnarray}
(\widetilde{\nabla}_{(X\wedge_S\xi)U}\sigma)(V,\xi)&=&-\sigma(\nabla_{(X\wedge_\xi)U}\xi,V)=-\sigma(\nabla_{S(\xi,U)X-S(X,U)\xi}\xi,V)\nonumber\\
&=&-S(\xi,U)\sigma(\nabla_X\xi,V)=S(\xi,U)\sigma(\alpha\phi{X}+\beta^2\phi{X},V)\nonumber\\
&=&S(\xi,U)[\alpha\phi\sigma(X,V)+\beta\sigma(X,V)].\label{24}
\end{eqnarray}
For the non-zero components of the second term, we have
\begin{eqnarray}
(\widetilde{\nabla}_{(X\wedge_S\xi)U}\sigma)(V,\xi)&=&-\sigma(\nabla_U,S(\xi,V)X-S(X,V)\xi)=-S(\xi,V)\sigma(\nabla_U\xi,X)\nonumber\\&=&S(\xi,V)\sigma(\alpha\phi{U}+\beta\phi^2U,X)\nonumber\\
&=&S(\xi,V)[\alpha\phi\sigma(U,X)+\beta\sigma(U,X)].\label{25}
\end{eqnarray}
Finally,
\begin{eqnarray}
(\widetilde{\nabla}_U\sigma)(V,(X\wedge_S\xi)\xi)&=&(\widetilde{\nabla}_U\sigma)(V,S(\xi,\xi)X-S(X,\xi)\xi)=(\widetilde{\nabla}_U\sigma)(S(\xi,\xi)X,V)\nonumber\\
&-&(\widetilde{\nabla}_U\sigma)(S(X,\xi)\xi,V)\nonumber\\
&=&(\widetilde{\nabla}_U\sigma)(S(\xi,\xi)X,V)-S(\xi,X)\sigma(\nabla_U\xi,V)\nonumber\\
&=&(\widetilde{\nabla}_U\sigma)(S(\xi,\xi)X,V)+S(X,\xi)[\alpha\phi\sigma(U,V)+\beta\sigma(U,V)].\label{26}
\end{eqnarray}
For $V=\xi$, (\ref{24}), (\ref{25}) and (\ref{26}) are put in (\ref{23}), we have
\begin{eqnarray*}
S(\xi,\xi)[\alpha\phi\sigma(X,U)&+&\beta\sigma(X,U)]+(\widetilde{\nabla}_U\sigma)(S(\xi,\xi)X,\xi)\nonumber\\
&=&S(\xi,\xi)[\alpha\phi\sigma(X,U)+\beta\sigma(X,U)]-\sigma(\nabla_U\xi,S(\xi,\xi)X)\nonumber\\
&=&0.
\end{eqnarray*}
Also taking into account (\ref{9}) and (\ref{20}), we obtain
\begin{eqnarray}
-2n(\alpha^2+\beta^2-\xi(\beta))[\alpha\phi\sigma(X,U)+\beta\sigma(X,U)]=0.\label{27}
\end{eqnarray}
Now  applying $\phi$ to (\ref{27}) and using (\ref{20}), we have
\begin{eqnarray}
-2n(\alpha^2+\beta^2-\xi(\beta))[\alpha\sigma(X,U)+\beta\phi\sigma(X,U)]=0.\label{28}
\end{eqnarray}
From (\ref{27}) and (\ref{28}), we can derive
\begin{eqnarray*}
-2n[\alpha^2+\beta^2-\xi(\beta)](\alpha^2-\beta^2)\sigma(U,V)=0.
\end{eqnarray*}
This proves our assertions.
\end{proof}
\begin{theorem}
	Let $M$ be an invariant submanifold of a lorentzian trans-Sasakian manifold $\widetilde{M}$. If $Q(g,\widetilde{\nabla}\cdot\sigma)=0$, then $M$ is either totally geodesic submanifold or $\alpha^2-\beta^2=0$.
\end{theorem}
\begin{proof}
	\begin{eqnarray*}
		Q(g,\widetilde{\nabla}\cdot\sigma)(U,V,Z;X,Y)=(\widetilde{\nabla}_U\sigma)((X\wedge_gY)V,Z)+(\widetilde{\nabla}_U\sigma)(V,(X\wedge_gY)Z)=0,
	\end{eqnarray*}
	for all $X,Y,U,V,Z\in\Gamma(TM)$. Taking $Y=V=\xi$ in last equality and by means of (\ref{13}) and (\ref{18}), we obtain
	\begin{eqnarray*}
		-(\widetilde{\nabla}_U\sigma)(X,Z)+\eta(X)\sigma(\nabla_U\eta(X)\xi,Z)-\sigma(\nabla_U\xi,\eta(Z)X)=0,
\end{eqnarray*}
	or
\begin{eqnarray*}	-(\widetilde{\nabla}_U\sigma)(X,Z)-\eta(X)\sigma(\alpha\phi{U}+\beta\phi^2U,Z)+\eta(Z)\sigma(\alpha\phi{U}+\beta\phi^2U,X)=0,
\end{eqnarray*}
which implies for $Z=\xi$
\begin{eqnarray}
\sigma(\nabla_U\xi,X)-\sigma(\alpha\phi{U}+\beta\phi^2U,X)&=&-2(\alpha\sigma(U,X)+\beta\sigma(U,X))\nonumber\\
&=&0.\label{29}
\end{eqnarray}
Applying $\phi$ to (\ref{29}) and we consider (\ref{20}), we have
\begin{eqnarray}
\alpha\phi\sigma(U,X)+\beta\sigma(U,X)=0.\label{30}
\end{eqnarray}
From (\ref{29}) and (\ref{30}) we conclude
\begin{eqnarray*}
(\alpha^2-\beta^2)\sigma(U,X)=0.
\end{eqnarray*}
This completes the proof.
\end{proof}
\begin{theorem}
 let $M$ be an invariant submanifold of lorentzian trans-Sasakian manifold $\widetilde{M}(\phi,\xi,\eta,g)$. If $Q(g,\widetilde{R}\cdot\sigma)=0$, then at least one of the following holds:\\
 1. $M$ is a totally geodesic submanifold,\\
 2. $\eta(\nabla(\beta-\alpha))=(\beta-\alpha)^2$\\
 3. $\eta(\nabla(\beta+\alpha))=(\beta+\alpha)^2$
\end{theorem}
\begin{proof}
 $Q(g,\widetilde{R}\cdot\sigma)=0$ leads to
 \begin{eqnarray*}
 (\widetilde{R}(X,Y)\cdot\sigma)((U\wedge_gV)Z,W)+(\widetilde{R}(X,Y)\cdot\sigma)(Z,(U\wedge_gV)Z)=0,
 \end{eqnarray*}
for all $X,Y,Z,U,V,W\in\Gamma(T\widetilde{M})$. For $Z=U=W=\xi$, this equality implies
\begin{eqnarray*}
&&(\widetilde{R}(X,Y)\cdot\sigma)(\eta(V)\xi+V,\xi)=R^\bot(X,Y)\sigma(\eta(V)\xi,\xi)-\sigma(R(X,Y)\eta(V)\xi,\xi)\nonumber\\
&-&\sigma(\eta(V)\xi,R(X,Y)\xi)+R^\bot(X,Y)\sigma(V,\xi)-\sigma(R(X,Y)V,\xi)-\sigma(V,R(X,Y)\xi)\nonumber\\
&=&0.
\end{eqnarray*}
Taking into account of (\ref{6}), (\ref{19}) and (\ref{20}),
we can infer
\begin{eqnarray}
&&\sigma(V,(\alpha^2+\beta^2)[\eta(Y)X-\eta(X)Y]+2\alpha\beta[\eta(Y)\phi{X}-\eta(X)\phi{Y}])\nonumber\\
&+&\sigma(V,Y(\alpha)\phi{X}-X(\alpha)\phi{Y}+Y(\beta)\phi^2X-X(\beta)\phi^2Y)=0.\label{31}
\end{eqnarray}
Also, if $\xi$ is taken instead of $X$ in (\ref{31})and by means of (\ref{1}) and (\ref{21}), we reach at
\begin{eqnarray}
[\alpha^2+\beta^2-\xi(\beta)]\sigma(V,Y)+2\alpha\beta\phi\sigma(V,Y)=0.\label{32}
\end{eqnarray}
Applying $\phi$ to (\ref{32}) and taking into accont that (\ref{20}), we have
\begin{eqnarray}
[\alpha^2+\beta^2-\xi(\beta)]\phi\sigma(V,Y)+2\alpha\beta\sigma(V,Y)=0.\label{33}
\end{eqnarray}
by combining (\ref{32}) and (\ref{33}), we observe
\begin{eqnarray*}
 \{[\alpha^2+\beta^2-\xi(\beta)]^2-[2\alpha\beta-\xi(\alpha)]^2\}\sigma(Y,V)=0.
\end{eqnarray*}
This completes the proof.
\end{proof}

 \begin{theorem}
 Let $M$ be an invariant submanifold of a lorentzian trans-Sasakian manifold $\widetilde{M}$. If $Q(S,\widetilde{R}\cdot\sigma)=0$, then at least one of the following holds;\\
 	1.) $M$ is a totally geodesic,\\
 	2.) $\xi(\beta)=\alpha^2+\beta^2$\\
 	3.) $\xi(\beta-\alpha)=(\alpha-\beta)^2$.\\
 	4.) $\xi(\beta+\alpha)=(\alpha+\beta)^2$.
 \end{theorem}
\begin{proof}
$Q(S,\widetilde{R}\cdot\sigma)=0$ has the form
\begin{eqnarray*}
Q(S,\widetilde{R}\cdot\sigma)(U,V,Z,W;X,Y)&=&(\widetilde{R}(X,Y)\cdot\sigma)((U\wedge_SV)Z,W)\nonumber\\
&+&(\widetilde{R}(X,Y)\cdot\sigma)(Z,(U\wedge_SV)W)=0.
\end{eqnarray*}
For $X=U=W=\xi$, this yields to
\begin{eqnarray*}
(\widetilde{R}(\xi,Y)\cdot\sigma)(S(V,Z)\xi-S(\xi,Z)V,\xi)&+&(\widetilde{R}(\xi,Y)\cdot\sigma)(Z,S(V,\xi)\xi-S(\xi,\xi)V)\nonumber\\&=&0.
\end{eqnarray*}
Non-zero components of this expansion give us
\begin{eqnarray}
S(\xi,Z)\sigma(V,R(\xi,Y)\xi)&-&S(V,\xi)\sigma(Z,R(\xi,Y)\xi)-R^\bot(\xi,Y)S(\xi,\xi)\sigma(V,Z)\nonumber\\
&&+S(\xi,\xi)\sigma(R(\xi,Y)V,Z)+S(\xi,\xi)\sigma(R(\xi,Y)Z,V)\nonumber\\
&=&0.\label{34}
\end{eqnarray}
Taking $Z=\xi$ in (\ref{34}) and by using (\ref{8}) and (\ref{21}), we verify
\begin{eqnarray*}
S(\xi,\xi)\sigma(R(\xi,Y)\xi,V)&=&\left\lbrace-[2n(\alpha^2+\beta^2)-\xi(\beta)]+(2n-1)\xi(\beta)\right\rbrace\nonumber\\
&&\otimes\sigma(V,(\alpha^2+\beta^2-\xi(\beta))\phi^2Y+(2\alpha\beta-\xi(\alpha))\phi{Y})=0.
\end{eqnarray*}
This is equivalent to
\begin{eqnarray}
2n(\alpha^2+\beta^2-\xi(\beta))\left[(\alpha^2+\beta^2-\xi(\beta))\sigma(V,Y)+(2\alpha\beta-\xi(\alpha))\phi\sigma(V,Y)\right]\nonumber\\
&=&0.\label{35}
\end{eqnarray}
If $Y$ is taken instead of $\phi{Y}$ in (\ref{35}) and using (\ref{20}), we obtain
\begin{eqnarray*}
2n(\alpha^2+\beta^2-\xi(\beta))\left[(\alpha^2+\beta^2-\xi(\beta))\phi\sigma(V,Y)+(2\alpha\beta-\xi(\alpha))\sigma(V,Y)\right]=0.
\end{eqnarray*}
From the last two-equalities, we conclude that
\begin{eqnarray*}
2n(\alpha^2+\beta^2-\xi(\beta))[(\alpha^2+\beta^2-\xi(\beta))^2-(2\alpha\beta-\xi(\alpha))^2]\sigma(Y,V)=0,
\end{eqnarray*}
which proves our assertion.
\end{proof}
\begin{theorem}
Let $M$ be an invariant submanifold of a lorentzian trans-Sasakian manifold $\widetilde{M}$. If $Q(g,\mathcal{C}\cdot\sigma)=0$, then at least one of the following holds;\\
1. $M$ is totally geodesic submanifold,\\
2. The scalar curvture $\tau$ of $\widetilde{M}$ satisfies
$\tau=2n(2n+1)[(\alpha\pm\beta)^2-\xi(\alpha\pm\beta)]$
\end{theorem}
\begin{proof}
$Q(g,\mathcal{C}\cdot\sigma)=0$ is of
\begin{eqnarray*}
(\mathcal{C}(X,Y)\cdot\sigma)((U\wedge_gV)Z,W)+(\mathcal{C}(X,Y)\cdot\sigma)(Z,(U\wedge_gV)Z)=0,
\end{eqnarray*}
for all $X,Y,Z,U,V,W\in\Gamma(T\widetilde{M})$.\\
Here, taking $Z=U=W=\xi$, we have
\begin{eqnarray*}
&&(\mathcal{C}(\xi,Y)\cdot\sigma)(g(V,Z)\xi-\eta(Z)V,\xi)+(\mathcal{C}(\xi,Y)\cdot\sigma)(Z,\eta(V)\xi+V)\nonumber\\
&=&(\mathcal{C}(\xi,Y)\cdot\sigma)(g(V,Z)\xi,\xi)-(\mathcal{C}(\xi,Y)\cdot\sigma)(\eta(Z)V,\xi)\nonumber\\
&+&(\mathcal{C}(\xi,Y)\cdot\sigma)(Z,\eta(V)\xi)+(\mathcal{C}(\xi,Y)\cdot\sigma)(g(V,Z)\xi,\xi)(Z,V)=0.
\end{eqnarray*}
As non-zero components in these expansions, one can easily to see
\begin{eqnarray}
\eta(Z)\sigma(V,\mathcal{C}(\xi,Y)\xi)&+&R^\bot(\xi,Y)\sigma(Z,V)-\sigma(\mathcal{C}(\xi,Y)Z,V)\nonumber\\
&-&\sigma(Z,\mathcal{C}(\xi,Y)V)=0.\label{36}
\end{eqnarray}
In (\ref{36}), setting $Z=\xi$ and consider (\ref{6}), (\ref{16}), (\ref{21}) we reach at
\begin{eqnarray}
\sigma(V,\mathcal{C}(\xi,Y)\xi)&=&\left( \alpha^2+\beta^2-\xi(\beta)-\frac{\tau}{2n(2n+1)}\right)\sigma(V,Y)\nonumber\\
&+&[2\alpha\beta-\xi(\alpha)]\phi\sigma(V,Y)=0.\label{37}
\end{eqnarray}
Replacing $\phi{Y}$ by $Y$ in the last inequality and making the necessary revisions, we get
\begin{eqnarray}
\left(\alpha^2+\beta^2-\xi(\beta)-\frac{\tau}{2n(2n+1)}\right)\phi\sigma(V,Y)&+&[2\alpha\beta-\xi(\alpha)]\sigma(V,Y)\nonumber\\
&=&0.\label{38}
\end{eqnarray}
Thus (\ref{37})  and (\ref{38}) give us
\begin{eqnarray*}
\left(\left[\alpha^2+\beta^2-\xi(\beta)-\frac{\tau}{2n(2n+1)} \right]^2-\left[2\alpha\beta-\xi(\alpha)\right]^2\right)\sigma(V,Y)=0,
\end{eqnarray*}
which proves our assertions.
\end{proof}
\begin{theorem}
Let $M$ be an invariant submanifold of a lorentzian trans-Sasakian manifold $\widetilde{M}$. If $Q(S,\mathcal{C}\cdot\sigma)=0$, then at least one of the following holds;\\
1. $M$ is totally geodesic submanifold,\\
2.The scalar curvture $\tau$ of $\widetilde{M}$ satisfies $\tau=\pm2n(2n+1)(2\alpha\beta-\xi(\alpha))$ or
$\tau=2n(2n+1)[(\alpha\pm\beta)^2-\xi(\alpha\pm\beta)]$.
\end{theorem}
\begin{proof}
$Q(S,\mathcal{C}\cdot\sigma)=0$ is of the form
\begin{eqnarray*}
(\mathcal{C}(X,Y)\cdot\sigma)((U\wedge_SV)Z,W)+(\mathcal{C}(X,Y)\cdot\sigma)(Z,(U\wedge_SV)Z)=0,
\end{eqnarray*}
for all $X,Y,U,V,Z\in\Gamma(T\widetilde{M})$.\\
Also these decompositions give us for $Z=U=W=\xi$
\begin{eqnarray*}
&&(\mathcal{C}(\xi,Y)\cdot\sigma)(S(V,Z)\xi-S(Z,\xi)V,\xi)+(\mathcal{C}(\xi,Y)\cdot\sigma)(Z,S(V,\xi)-S(\xi,\xi)V)\nonumber\\
&=&(\mathcal{C}(\xi,Y)\cdot\sigma)(S(V,Z)\xi,\xi)-(\mathcal{C}(\xi,Y)\cdot\sigma)(S(Z,\xi)V,\xi)\nonumber\\
&+&(\mathcal{C}(\xi,Y)\cdot\sigma)(Z,S(V,\xi)\xi)-(\mathcal{C}(\xi,Y)\cdot\sigma)(S(\xi,\xi)V,Z)=0.
\end{eqnarray*}
From these expressions are expansioned, we have non-zero components
\begin{eqnarray}
&&S(\xi,Z)\sigma(V,\mathcal{C}(\xi,Y)\xi)-S(\xi,V)\sigma(Z,\mathcal{C}(\xi,Y)\xi)-R^\bot(\xi,Y)S(\xi,\xi)\sigma(V,Z)\nonumber\\
&&+S(\xi,\xi)\sigma(\mathcal{C}(\xi,Y)Z,V)+S(\xi,\xi)\sigma(Z,\mathcal{C}(\xi,Y)V)=0.\label{39}
\end{eqnarray}
Next, replacing $\xi$ by $Z$ in (\ref{39}), by using (\ref{6}) and (\ref{16}), we can see
\begin{eqnarray*}
&&S(\xi,\xi)\sigma(\mathcal{C}(\xi,Y)\xi,V)=S(\xi,\xi)(\alpha^2+\beta^2-\xi(\beta)-\frac{\tau}{2n(2n+1)})\sigma(V,Y)+\nonumber\\
&+&S(\xi,\xi)(2\alpha\beta-\xi(\alpha))\phi\sigma(V,Y)=0,
\end{eqnarray*}
that is,
\begin{eqnarray*}
&&2n(\alpha^2+\beta^2-\xi(\beta))\left\lbrace[\alpha^2+\beta^2-\xi(\beta-\frac{\tau}{2n(2n+1)})]^2-[2\alpha\beta-\xi(\alpha)]^2 \right\rbrace\sigma(V,Y)\nonumber\\
&&=0.
\end{eqnarray*}
This completes the proof.
\end{proof}
Nex, we will now build an example to illustrate our topic.
\begin{example}
Let
$\widetilde{M}^5=\{(x_1,x_2,x_3,x_4,t)\in\textbf{E}^5:t\neq{0}\}$be
a 5-dimensional differentiable manifold with the standart coordinate
system $(x_1,x_2,x_3,x_4,t)$. Then vector fields
\begin{eqnarray*}
E_1&=&t(\frac{\partial}{\partial{x_1}}+x_2\frac{\partial}{\partial{t}}), \ \ E_2=t\frac{\partial}{\partial{x_2}}, \ \  E_3=t(\frac{\partial}{\partial{x_3}}+x_4\frac{\partial}{\partial{t}}),\\
E_4&=&t\frac{\partial}{\partial{x_4}},\ \
E_5=\xi=\frac{\partial}{\partial{t}}
\end{eqnarray*}
are the linear independent at each points of $\widetilde{M}$, that
is, these vector fields a basis of tangent space of $\widetilde{M}$.
Now, we define the contact structure and metric tensor $\phi$ and
$g$ by
\begin{eqnarray*}
\phi{E_1}=-E_2,\ \  \phi{E_2}=E_1,\ \ \phi{E_3}=-E_4,\ \
\phi{E_4}=E_3,\ \ \phi{E_5}=0,
\end{eqnarray*}
and
\begin{eqnarray*}
g(E_i,E_j)&=&\delta_{ij},\ \ 1\leq{i,j}\leq{4}\nonumber\\
g(E_5,E_5)&=&-1
\end{eqnarray*}
then we can easily verify that
\begin{eqnarray*}
\phi^2X=X+\eta(X)\xi, \ \ g(\phi{X},\phi{Y})=g(X,Y)+\eta(X)\eta(Y),
\ \ \eta(X)=g(X,\xi).
\end{eqnarray*}
So, $\widetilde{M}^5(\phi,\xi,\eta,g)$ is a 5-dimensional contact
metric manifold. By direct calculations, we can get the non-zero
components of Lie-bracket as
\begin{eqnarray*}
[E_i,E_5]=-\frac{1}{t}E_i, \ \ 1\leq{i}\leq{4},
\end{eqnarray*}
\begin{eqnarray*}
[E_1,E_2]=x_2E_2-t^2E_5, \ \ [E_1,E_3]=-x_4E_1+x_2E_3, \ \
[E_1,E_4]=x_2E_4.
\end{eqnarray*}
In view of Kozsul formulae, we can find the following non-zero
components of connections as
\begin{eqnarray*}
\widetilde{\nabla}_{E_1}E_5=-\frac{1}{t}E_1+\frac{1}{2}t^2E_2, \ \widetilde{\nabla}_{E_2}E_5=-\frac{1}{2}t^2E_1-\frac{1}{t}E_2\\
\widetilde{\nabla}_{E_3}E_5=-\frac{1}{t}E_3+\frac{1}{2}t^2E_4, \ \
\widetilde{\nabla}_{E_4}E_5=-\frac{1}{2}t^2E_3-\frac{1}{t}E_4.
\end{eqnarray*}
By the straightforward calculations, by using (\ref{5}),we can observe $\alpha=\frac{1}{2}t^2$ and $\beta=-\frac{1}{t}$. This tell us that $\widetilde{M}^5(\phi,\xi,\eta,g)$ is a 5-dimensional trans-Sasakain manifold with $\alpha=\frac{1}{2}t^2$ and $\beta=-\frac{1}{t}$.\\
    Now, we consider vector fields
    \begin{eqnarray*}
        e_1=t(\frac{\partial}{\partial{x_1}}+\frac{\partial}{\partial{x_3}}+(x_2+x_4)\frac{\partial}{\partial{t}}), \ \ e_2=t(\frac{\partial}{\partial{x_2}}+\frac{\partial}{\partial{x_4}}), \ \ e_3=\frac{\partial}{\partial{t}}.
    \end{eqnarray*}
    These vector fields are the linear dependent. By $\textbf{D}$, let's denote the distribution spanned by these vectors. One can observe $\textbf{D}$ is integrable and involutive. By $M$ we denote the its integral manifold, then $M$ manifold is a submanifold of $\widetilde{M}^5(\phi,\xi,\eta,g)$. On can easily to see that
    \begin{eqnarray*}
        \phi{e_1}=-e_2, \ \ \phi{e_2}=e_1, \ \ and \ \ \phi{e_3}=0.
    \end{eqnarray*}
    This tell us that $M$ is an invariant submanifold of a trans-Sasakaim manifold $\widetilde{M}^5(\phi,\xi,\eta,g)$. On the other hand, by direct calculations, we verify that $\sigma(e_1,e_2)=0$, that is, $M$ is totally geodesic submanifold and $\xi(\alpha)=t$, $\xi(\beta)=\frac{1}{t^2}$.
\end{example}

\end{document}